\documentclass{aims} 
\usepackage{amsmath}

\usepackage{paralist}

\usepackage[misc]{ifsym}
\usepackage{epsfig} 
\usepackage{epstopdf} 
\usepackage[colorlinks=true]{hyperref}
\hypersetup{urlcolor=blue, citecolor=red}
\usepackage{latexsym,amssymb,graphicx,upgreek}
\usepackage[mathscr]{euscript}
\usepackage{tcolorbox}
\usepackage{algorithm}
\usepackage{algorithmic}
\usepackage{xcolor}
\usepackage{enumitem}
\usepackage{mathtools}
\usepackage{caption}

\allowdisplaybreaks

\usepackage{titlesec}
\titlespacing*{\section}
  {0pt}    
  {2.0ex plus 0.5ex minus 0.2ex}  
  {1.0ex plus 0.2ex}     

\titlespacing*{\subsection}
  {0pt}
  {1.5ex plus 0.4ex minus 0.2ex}
  {0.7ex plus 0.2ex}

\theoremstyle{plain}
\newtheorem{theorem}{Theorem}[section]

\newtheorem{lemma}[theorem]{Lemma}
\newtheorem{proposition}[theorem]{Proposition}

\newtheorem{remark}[theorem]{Remark}
\theoremstyle{definition}

\DeclareMathOperator*{\Range}{Range}
\DeclareMathOperator*{\argmax}{arg\,max}

\newcommand{\dhat}{\widehat{\Delta}}

\newcommand{\M}{\mathcal{M}}
\renewcommand{\H}{\mathrm{H}}
\newcommand{\Ht}{\tilde{\mathrm{H}}}
\newcommand{\Ft}{\tilde{\mathrm{F}}}
\newcommand{\f}{\mathrm{f}}
\newcommand{\ft}{\tilde{\mathrm{f}}}
\newcommand{\F}{\mathrm{F}}

\newcommand{\I}{\mathrm{I}}

\newcommand{\G}{\mathrm{G}}
\newcommand{\A}{\mathrm{A}}
\newcommand{\B}{\mathrm{B}}
\newcommand{\C}{\mathrm{C}}
\newcommand{\K}{\mathrm{K}}

\newcommand{\postm}{\mu_\text{post}^y}
\newcommand{\priorm}{\mu_\text{pr}}

\newcommand{\R}{\mathbb{R}}
\newcommand{\GM}[2]{\mathrm{N}({#1}, {#2})}

\newcommand{\ipm}[2]{\langle {#1}, {#2}\rangle}
\newcommand{\ipar}{m}
\newcommand{\iparpr}{m_\mathrm{pr}}
\newcommand{\iparmap}{m_\mathrm{MAP}}
\newcommand{\ncov}{\Gamma_{\mathrm{noise}}}
\newcommand{\Cprior}{\C_\mathrm{pr}}
\newcommand{\Cpost}{\C_\mathrm{post}}
\newcommand{\Cinc}[1]{\C^{[{#1}]}_\mathrm{post}}
\newcommand{\defeq}{\vcentcolon=}
\newcommand{\deriv}[3]{\Delta_{{#1}}({#3}\,|\,{#2})}
\newcommand{\Ns}{d}
\newcommand{\ave}[2]{\mathsf{E}_{{#1}}\!\left\{ {#2} \right\}}
\newcommand{\obs}{y}

\newcommand{\avey}[1]{\mathsf{E}_{{\obs|\ipar}}\left\{ {#1} \right\}}
\newcommand{\DKL}{\mathrm{D}_{\mathrm{kl}}}
\newcommand{\pEIG}{\Phi_{\mathrm{eig}}}
\newcommand{\EIG}{\mathsf{EIG}}
\newcommand{\trace}{\mathrm{trace}}
\renewcommand{\L}{\mathscr{L}}

\newcommand{\Lsymp}{\mathscr{L}_{\text{sym}}^{+}}

\newcommand{\arr}{\upalpha_{rr}}
\newcommand{\att}{\upalpha_{tt}}
\newcommand{\art}{\upalpha_{rt}}

\title[Submodularity of EIG in infinite dimensions]
{Submodularity of the expected information gain 
in infinite-dimensional linear inverse problems}  
\author[Alen Alexanderian and Steven Maio]{}

\thanks{$^*$Corresponding author: Alen Alexanderian (\texttt{alexanderian@ncsu.edu})}

\begin{document}

\maketitle
\centerline{\scshape
Alen Alexanderian$^{{\href{mailto:alexanderian@ncsu.edu}{\textrm{\Letter}}}*1}$
and Steven Maio$^{{\href{mailto:smaio@ncsu.edu}{\textrm{\Letter}}}1}$}

\medskip

{\footnotesize
 \centerline{$^1$Department of Mathematics, North Carolina State University, USA}
} 

\bigskip

\begin{abstract}
We consider infinite-dimensional linear Gaussian Bayesian inverse problems with
uncorrelated measurement errors and focus on the problem of selecting sensor
placements that maximize the expected information gain (EIG).  This study is
motivated by optimal sensor placement for linear inverse problems constrained by
partial differential equations (PDEs).  We consider measurement models where
each sensor collects a single-snapshot measurement.  This covers sensor
placement for inverse problems governed by linear steady PDEs or evolution
equations with final-in-time observations.  It is well-known that in the
finite-dimensional (discretized) formulations of such inverse problems, the EIG is a
monotone submodular function. This also entails a theoretical guarantee for
greedy sensor placement in the discretized setting.  We extend the result on
submodularity of the EIG to the infinite-dimensional setting, proving that the
approximation guarantee of greedy sensor placement remains valid in the
infinite-dimensional limit.  We also discuss computational considerations and
present strategies that exploit problem structure and submodularity to yield 
efficient implementations of the greedy procedure.  
\end{abstract}

\section{Introduction}
Inverse problems constrained by partial differential equations (PDEs) are common
in applications.  We focus on linear inverse problems.
Examples of such problems include source inversion problems governed by linear
steady or time-dependent PDEs.  
We assume that the inversion parameters of
interest take values in an infinite-dimensional real separable Hilbert space
$\M$ and that data are collected from a set of sensors.
Our main goal is to show that, for infinite-dimensional linear Gaussian Bayesian
inverse problems with uncorrelated measurement errors, the expected information
gain (EIG) is a monotone submodular function.

Consider a generic linear inverse problem of estimating 
$m \in \M$ using the observation model,
\begin{equation}\label{equ:model}
\obs = \F m + \eta.
\end{equation}
Here, $\F:\M \to \R^d$ is a continuous linear transformation, $\obs$ is data,
and $\eta \in \R^d$ models observation error.  We let $\eta \sim \GM{0}{\ncov}$,
with $\ncov =
\mathrm{diag}(\sigma_1^2, \ldots, \sigma_\Ns^2)$.  
The operator $\F$ is called the parameter-to-observable map. 
Here, $\F$ maps
the inversion parameter $m$ to model output at $d$ candidate sensor locations.
In inverse problems governed by PDEs, computing $\F m$ requires solving the
governing PDE and extracting solution values at measurement points.  In such
problems, optimal sensor placement is crucial---a suboptimal placement of
sensors can result in waste of experimental resources and highly inaccurate
parameter estimates.  Optimal sensor placement can be formulated as an optimal
experimental design (OED) 
problem~\cite{Alexanderian21,Ucinski05}.  
The fundamental question is how should we place the sensors to maximize the
information gain about the parameter $m$.

The abstract Hilbert space setting considered herein covers a broad range of 
PDE-constrained linear inverse problems. One example is source inversion in a
steady advection-diffusion equation, where sensor measurements of the state
variable $u$ are used to estimate the source term $m$ in the following boundary 
value problem:
\begin{equation}
\begin{alignedat}{2}
-\kappa \Delta u + \nabla \cdot (v u) &= m \quad &\text{in } &\Omega, \\
u &= 0 \quad &\text{on } &\Gamma_\text{D},\\
\kappa \nabla u \cdot n &= 0 &\quad \text{on } &\Gamma_\text{N}. 
\end{alignedat}
\end{equation} 
Here, $\Omega \subset \R^2$ is a bounded domain, $\kappa > 0$ is a diffusion
constant, and $v$ is a prescribed velocity field. The boundary is partitioned as
$\partial\Omega = \Gamma_\text{D} \cup \Gamma_\text{N}$ with $\Gamma_\text{D}
\cap \Gamma_\text{N} = \emptyset$.  In this example, the OED problem seeks a
placement of sensors at which measurements of $u$ are taken so as to
maximize the information gain about the source term $m$.

In the rest of the article, we focus on the general linear inverse problem of
estimating $m \in \M$ from the observation model~\eqref{equ:model}.
Considering the inverse problem within a Bayesian paradigm, we
assume a Gaussian prior 
$\priorm = \GM{\iparpr}{\Cprior}$, with $\Cprior$ a strictly 
positive self-adjoint trace-class operator on $\M$ and $\iparpr$ a sufficiently 
regular element of $\M$. This prior measure encodes our prior state of 
knowledge regarding the inversion parameter before collecting measurement data.
In the present linear Gaussian problem setup, it is known that the posterior
measure $\postm$, which describes the distribution law of the inversion parameter 
conditioned on observed data, is also Gaussian. Furthermore, this Gaussian 
posterior admits closed-form expressions for its
mean and covariance operator~\cite{Stuart10}; see Section~\ref{sec:bayesIP}.

We consider the OED problem of choosing an optimal subset from a set of candidate 
sensor locations subject to a cardinality constraint.  Specifically, consider
the problem of finding a sensor placement that maximizes the EIG, as measured by
the expected Kullback--Leibler (KL)
divergence~\cite{kullback:1951,ManningSchutze99} from $\postm$ to $\priorm$:
\begin{equation}\label{equ:avgKL} 
    \EIG \defeq  \ave{\priorm}{\avey{\DKL(\postm\| \priorm)}}.  
\end{equation} 
Here, $\DKL(\postm\| \priorm)$ denotes the KL 
divergence from $\postm$ to $\priorm$.
Namely, 
\[ 
    \DKL(\postm\| \priorm) \defeq \int_\M \log\left[\frac{d\postm}{d\priorm}\right]\, d\postm,  
\] 
which is well-defined 
due to the absolute continuity of $\postm$ with respect to $\priorm$.  
Also, in~\eqref{equ:avgKL}, $\mathbb{E}_{\priorm}$ and
$\mathbb{E}_{{\obs|\ipar}}$, respectively, denote expectation with respect to
$\priorm$ and the likelihood measure.

In the present linear Gaussian setting, $\EIG$ admits an analytic
expression, which also extends to infinite 
dimensions~\cite{AlexanderianGloorGhattas16,AlexanderianSaibaba18}; see
Section~\ref{sec:bayesIP} for details.  We consider the problem of finding an
optimal subset of size $k$ from a set of $\Ns$ candidate sensor
locations such that $\EIG$ is maximized.
This is a challenging optimization problem with combinatorial complexity. 
An exhaustive search for an optimal subset is impractical for even modest values of 
$\Ns$ and $k$. For example, with $\Ns = 100$ and $k = 20$, an exhaustive 
search for an optimal sensor placement requires  
${100\choose 20} = \mathcal{O}(10^{20})$ function evaluations. A standard approach 
for obtaining a near optimal solution is to follow a greedy strategy,
where sensors are chosen one at a time: at each step, we choose a sensor that 
results in the largest incremental increase in the value of $\EIG$. 

A greedy approach for finding near optimal sensor placements is a popular
approach and has long been used in finite-dimensional parameter estimation
problems.  This approach is simple to implement. It is also built on a solid
theoretical foundation and admits an approximation
guarantee~\cite{KrauseGolovin14,NemhauserWolseyFisher78}.  This is rooted in a
key property of $\EIG$ in the linear Gaussian setup:
\emph{submodularity}~\cite{KelmansKimelfeld83,
KrauseSinghGuestrin08,ShamaiahBanerjeeVikalo10}; see
Section~\ref{sec:submodular}.
However, the existing proofs
for submodularity of $\EIG$ consider only the finite-dimensional parameter estimation
problems~\cite{KrauseSinghGuestrin08,ShamaiahBanerjeeVikalo10}.
Greedy sensor placement for maximizing $\EIG$ has also been considered for
PDE-based inverse problems~\cite{AlexanderianDiazRaoEtAl2026,EswarRaoSaibaba2026}.  In practical
computations, one considers a discretized version of the inverse problem, in
which case the theoretical results regarding submodularity of $\EIG$ apply.

However, in the infinite-dimensional inverse problems under study, a rigorous
function-space formulation of the problem is essential.  This is important from
both theoretical and practical perspectives. Namely, a rigorous
infinite-dimensional formulation guides suitable choices for the prior
covariance operator, the definition of the OED objective, and consistent
discretization schemes for the different components of the problem.  It is
therefore of interest to establish the submodularity property of the EIG in an
infinite-dimensional setting as well.  To our knowledge, this has not been
considered in the literature.  

In this article, we prove that $\EIG$ is a monotone submodular function in
linear Gaussian inverse problems on infinite-dimensional real separable Hilbert
spaces.  Keeping the discussion in a function space setting leads to elegant
arguments and helps provide key structural insights that remain relevant for
finite-dimensional problems.  We also discuss computational considerations and
outline efficient implementations of the greedy procedure.  In particular, 
we incorporate
formulations that maximally exploit the problem structure and adopt a key tool
from submodular optimization---lazy greedy.  The latter is well-known in
combinatorial optimization and the computer science communities, but is
underutilized in sensor placement problems for infinite-dimensional inverse
problems governed by PDEs. 

To facilitate the proof of submodularity of $\EIG$, in
Section~\ref{sec:design}, we describe how a design is incorporated in the
formulation of the inverse problem and present a convenient representation of the
posterior covariance operator $\Cpost$. We also derive an update formula for $\Cpost$
and consider some monotonicity properties of $\Cpost$.  The discussion in
Section~\ref{sec:design} concludes by presenting a precise mathematical
description of the OED problem under study.  In Section~\ref{sec:main}, we prove
that $\EIG$ is a monotone submodular function.  This is followed by a discussion of
computational considerations regarding efficient implementations of a greedy
procedure in Section~\ref{sec:lazy_dopt}.  Concluding remarks are provided in
Section~\ref{sec:conc}.

\section{Preliminaries}
\label{sec:background}
In this section, we discuss the requisite notation and background 
concepts needed in the rest of the article.

\subsection{Basic notation and definitions}
As before, $\M$ is assumed to be an infinite-dimensional real separable Hilbert 
space.  The inner product on $\M$ is denoted by $\ipm{\cdot}{\cdot}$ and the 
corresponding induced norm is $\| \cdot \| = \ipm{\cdot}{\cdot}^{1/2}$.  We
denote the space of bounded linear operators on $\M$ by $\L(\M)$.  An operator
$\A \in \L(\M)$ is said to be \emph{invertible} if it has an inverse in
$\L(\M)$; see, e.g.,~\cite[Page 39]{MacCluer09}.  
For $\A \in \L(\M)$, its adjoint is denoted by $\A^*$. 

In what follows, we also need to work with positive self-adjoint 
operators and trace-class operators.
The set of positive self-adjoint 
bounded linear operators on $\M$ is denoted by $\Lsymp(\M)$: 
\[
\Lsymp(\M) \defeq \{ \A \in \L(\M) : \A = \A^* \text{ and }
\ipm{\A u}{u} \geq 0, \quad \text{for all } u \in \M\}.
\]

A linear operator $\A \in \L(\M)$ is said to be \emph{trace-class} if 
\begin{equation}\label{equ:traceclass_def}
\sum_{i=1}^\infty \ipm{(\A^*\A)^{1/2} e_i}{e_i} < \infty,
\end{equation}
for an orthonormal basis $\{e_i\}_{i=1}^\infty$ of $\M$. 
If~\eqref{equ:traceclass_def} holds for one orthonormal basis, then
it holds for any other orthonormal basis, and the value of the sum is
independent of the choice of basis~\cite[Chapter VI]{ReedSimon80}.
For a trace-class operator, its trace, 
\[
\trace(\A) \defeq \sum_{i=1}^\infty \ipm{\A e_i}{e_i}, 
\]
is finite and independent of the choice of an orthonormal basis.

For a positive self-adjoint trace-class operator $\K$ on $\M$, 
the \emph{Fredholm determinant}~\cite{Simon77,Simon05} of $\I+\K$ is given by
\[
\det(\I+\K) := \prod_{n=1}^\infty (1+\lambda_n),
\]
where $\{\lambda_n\}_{n=1}^{\infty}$ are the eigenvalues of $\K$, counted with
multiplicity.  The Fredholm determinant can also be defined for trace-class
operators that are not necessarily self-adjoint.  This is done via a
construction involving exterior powers of Hilbert spaces~\cite{Simon77}.  See
also~\cite{BoydSnigireva19}, which considers both real and complex Hilbert
spaces.  

Various familiar properties of the determinant from linear algebra 
can be extended to the case 
of the Fredholm determinant. In particular, 
we have the following multiplicative property: if $\K_1, \K_2$ are 
trace-class operators, then
\begin{equation}\label{equ:mult}
    \det\big((\I+\K_1)(\I+\K_2)\big) = \det(\I+\K_1)\det(\I+\K_2).
\end{equation}
In the discussions that follow, we work primarily with operators of the form
$\I+\K$ with $\K \in \Lsymp(\M)$ of finite rank, where the Fredholm determinant
behaves analogously to the determinant in finite dimensions.

When considering rank-one updates, we will be using the tensor product of elements in 
$\M$. For $u$ and $v$ in $\M$, their tensor product $u \otimes v \in \L(\M)$ is the operator 
\[
(u \otimes v)x \defeq \ipm{v}{x}u, \quad \text{for all } x \in \M.
\] 
We also recall the following identity: for $\A, \B \in \L(\M)$ and $u, v \in \M$, 
\[
\A (u \otimes v) \B = (\A u) \otimes (\B^* v).
\] 

\subsection{Bayesian linear inverse problems and optimal sensor placement}\label{sec:bayesIP}
In the linear Gaussian Bayesian inverse problem setup described in the introduction, 
the posterior measure is Gaussian,
$\postm = \GM{\iparmap}{\Cpost}$, with
\begin{equation}\label{equ:posterior}
\Cpost = (\F^*\ncov^{-1}\F + \Cprior^{-1})^{-1}
\quad \text{and} \quad
\iparmap = \Cpost(\F^*\ncov^{-1} \obs + \Cprior^{-1}\iparpr).
\end{equation}
See~\cite{Stuart10}, for
further details. The notation for the posterior mean is motivated by its 
variational characterization as the maximum a posteriori probability (MAP) point. 
Specifically, it can be shown~\cite{DashtiLawStuart13,Stuart10} 
that $\iparmap$ is the unique global minimizer of
\begin{equation}\label{equ:MAP_estimation}
J(m) \defeq \frac12 (\F m - y)^{\top} \ncov^{-1} (\F m - y) + 
\frac12 \| \Cprior^{-1/2}(m - \iparpr)\|^2
\end{equation}
in the Cameron--Martin space $\M_{\priorm} \defeq \Range(\Cprior^{1/2})$.

We next define the operators
\begin{equation}\label{equ:hessian_operators}
  \H \defeq \F^* \ncov^{-1}\F  \quad\text{and}\quad
  \Ht \defeq \Cprior^{1/2} \H \Cprior^{1/2}.
\end{equation}
These definitions are also linked to the variational characterization of $\iparmap$.
Namely, $\H$ is the Hessian of the data misfit term in the definition of 
the cost functional $J$
in~\eqref{equ:MAP_estimation}.
In this context, 
$\Ht$ is known as the \emph{prior-preconditioned data misfit Hessian}.
Using this notation, we can represent the posterior covariance operator as 
\begin{equation}\label{equ:Cpost_alt}
\Cpost = (\H + \Cprior^{-1})^{-1} = \Cprior^{1/2} (\I + \Ht)^{-1} \Cprior^{1/2}. 
\end{equation}

Optimal sensor placement is framed as the problem of identifying a placement of
sensors that maximizes the information content extracted from the collected
measurements.  As noted in the introduction, this can be formulated as an
optimal experimental design (OED) problem. This requires specifying a design
criterion that quantifies the quality of a sensor placement (experimental
design). Examples include the Bayesian A-optimality criterion, which
quantifies the average posterior variance, or the expected information gain
(EIG) defined in~\eqref{equ:avgKL}. 

In the present class of infinite-dimensional linear inverse 
problems, the EIG 
admits the following analytic expression~\cite{AlexanderianGloorGhattas16,AlexanderianSaibaba18}:
\begin{equation}\label{equ:EIG_def}
\EIG := \frac12\log\det(\I + \Ht). 
\end{equation} 
The corresponding OED problem seeks sensor placements that maximize $\EIG$.

\subsection{Submodular function optimization}\label{sec:submodular}
Let $V = \{1, \ldots, \Ns\}$ be finite set, and consider 
a function $f: 2^{V}\to\mathbb{R}$. 
In the present setting, $V$ is often referred to as the \emph{ground set} and 
indexes a set of experiments. Also, $2^V$ is the power set of $V$. 
The function $f$ assigns a utility to each subset of experiments.

For $S \subseteq V$, the marginal gain~\cite{KrauseGolovin14} 
of adding an element $v \in V$ to $S$ is defined as
\[
\deriv{f}{S}{v} \defeq f(S \cup \{ v\}) - f(S).
\]
The function $f$ is said to be 
\begin{itemize}
    \item \textit{monotone}, if $\deriv{f}{S}{v} \ge 0 $ for every $S \subset V$ and $v \in V \backslash S$;
\item  
\textit{submodular}, if for every $A \subseteq B \subseteq V$ and $v \in V\setminus B$,
we have 
\[
    \deriv{f}{A}{v} \ge \deriv{f}{B}{v}.
\]
\end{itemize}
Submodularity can be understood as a diminishing returns property.
In the context of experimental design, this says that conducting the experiments in 
$B \setminus A$ does not increase the utility of conducting the experiment indexed by $v \in 
V \setminus B$.

It is well-known~\cite[Theorem 44.1]{Schrijver03} that proving submodularity can be
reduced to the case of one element increments. Specifically, to prove $f$ is
submodular, we need to verify that for every $S \subseteq V$ and distinct $v,
w\in V\setminus S$
\begin{equation}\label{equ:incr-submodularity}
    \deriv{f}{S}{v} 
    \ge 
    \deriv{f}{S\cup\{w\}}{v}. 
\end{equation}

Consider now an optimization problem of the form 
\begin{equation}\label{equ:optim_binary}
\begin{aligned}
&\max_{S \subset V} f(S), \\
&\text{subject to } |S| = k.
\end{aligned}
\end{equation}
In the case of an optimal sensor placement problem, 
$V$ indexes a set of candidate sensor locations, $f$ is a utility function 
(e.g., the EIG), and $k$ is a sensor budget.

Maximizing a monotone submodular function subject to cardinality constraints is
a challenging combinatorial optimization problem.  A tractable and popular
approach to obtain approximate solutions to such optimization problems 
is the greedy algorithm, which  follows the following intuitive
procedure: Begin with $S = \emptyset$, and add elements of $V$ to $S$ one at
a time; in each step, pick an element that provides the largest marginal gain. 
This procedure is outlined in Algorithm~\ref{alg:greedy_submodular}.

For a monotone submodular function $f$, the greedy procedure 
applied to~\eqref{equ:optim_binary}
achieves an approximation ratio of $1 - e^{-1}$ with only $\mathcal{O}(kd)$ 
evaluations of $f$. That is, 
the output $S$ Algorithm~\ref{alg:greedy_submodular} satisfies
\begin{equation}\label{eq:approximation-ratio}
f(S) \ge \left(1 - 1/e \right)\max\left\{f(R) : |R|\le k\right\},
\end{equation}
where $e$ is the base of the natural logarithm;
see~\cite{KrauseGolovin14,NemhauserWolseyFisher78} for details.

\begin{algorithm}[H]
\caption{Greedy algorithm for submodular function maximization}
\label{alg:greedy_submodular}
\begin{algorithmic}[1]
\REQUIRE Ground set $V$, function $f : 2^V \to \R$, and budget $k$
\ENSURE Selected set $S \subseteq V$ with $|S| = k$
\STATE $S_0 = \emptyset$
\FOR{$l = 1,\dots,k$}
    \STATE $\displaystyle v^\ast = \argmax_{v \in V \setminus S_{l-1}} \deriv{f}{S_{l-1}}{v}$  
    \STATE $S_l = S_{l-1} \cup \{v^\ast\}$
\ENDFOR
\RETURN $S = S_k$
\end{algorithmic}
\end{algorithm}

The utility of greedy strategies for submodular optimization is not limited to
simple cardinality constraints. Suitable variants of the greedy approach
maintain rigorous approximation guarantees under more general types of
constraints relevant to optimal sensor placement. Here, we briefly describe two
important constraint types: \emph{matroid} and \emph{knapsack} constraints.

Matroid constraints are relevant in the context of sensor
scheduling~\cite{KrauseGolovin14}. In such problems one has a grid of available
sensors that can be activated as needed.  To conserve energy or operating costs,
the OED problem seeks to find optimal times for activating each sensor.  Matroid
constraints limit the number of sensors that can be active at any single
measurement time. This type of sensing strategy is also discussed
in~\cite[Chapter 4]{Ucinski05} from a different perspective, where the term
\emph{scanning sensors} is used to refer to such measurement setups.  

Knapsack constraints on the other hand allow elements of $V$ to have varying
costs, and constrain the total cost of a feasible subset to be below a given
budget.  This is relevant in sensor placement problems where different types of
sensors, which provide varying levels of accuracy or measure different physical
quantities, are used.

\section{Incorporating the experimental design in the inverse problem}\label{sec:design}
In this section, we consider how the design enters the formulation of the
inverse problem.  We index the set of candidate sensor sites by the set $V
= \{1, \ldots, \Ns\}$.  Subsequently, an experimental design (sensor placement)
is characterized by a subset $S \subset V$.  
We begin our discussion by deriving a convenient representation of the posterior
covariance operator as a function of $S \in 2^V$ in Section~\ref{sec:Cpost}.  
Subsequently, we consider the impact of adding a single sensor on the posterior
covariance operator and some monotonicity properties of $\Cpost$ in
Section~\ref{sec:update}.  Finally, we formulate the OED problem of finding
sensor placements that maximize the EIG in
Section~\ref{sec:eig_optim}.

\subsection{Design-dependent posterior covariance operator}\label{sec:Cpost}  
We begin by considering the following technical result, which provides a
convenient representation for the data misfit Hessian $\H$ in~\eqref{equ:hessian_operators} 
in terms of rank-one updates. 
This also plays a key role in proving the submodularity of
the EIG, as detailed in Section~\ref{sec:main}.  
\begin{lemma}\label{lem:Hrep}
The operator $\H$ admits the following representation: 
\begin{equation}\label{equ:Hrep}
\H = \sum_{i=1}^\Ns \f_i \otimes \f_i, 
\end{equation}
where $\f_i = \sigma_i^{-1} \F^* e_i$, with $\{e_i\}_{i=1}^d$ the standard 
basis in $\R^d$ and $\F^*$ the adjoint of $\F$.
\end{lemma}
\begin{proof}
This follows by noting that 
\[
\F^*\ncov^{-1} \F 
= \F^* \Big(\sum_{i=1}^d \sigma_i^{-2} e_i \otimes e_i \Big)\F 
= \sum_{i=1}^d (\sigma_i^{-1}\F^* e_i) \otimes (\sigma_i^{-1}\F^* e_i). \qedhere
\]
\end{proof}
The representation~\eqref{equ:Hrep} enables defining 
$\H$ and $\Ht$ as a function of $S \in 2^V$, where $V = \{1, \ldots,
d\}$ indexes the set of all candidate sensor locations. Namely, 
$\H(S) = \sum_{i \in S} \f_i \otimes \f_i$ and 
\begin{equation}\label{equ:Ht_rep}
\Ht(S) = \sum_{i \in S} \ft_i \otimes \ft_i, \quad
\text{where}\quad \ft_i = \Cprior^{1/2} \f_i. 
\end{equation}
Also, recalling the expression for the posterior covariance operator in~\eqref{equ:Cpost_alt},
we have
\begin{equation}\label{equ:Crep}
\Cpost(S) =  \Cprior^{1/2} \big(\I + \Ht(S)\big)^{-1} \Cprior^{1/2}.  
\end{equation}

Considering the observation model~\eqref{equ:model}, 
it is straightforward to see that 
\begin{equation}\label{equ:yi}
   y_i = \ipm{\F^* e_i}{m} + \eta_i, \quad i \in \{1, \ldots, d\},
\end{equation}
where $\{e_i\}_{i=1}^d$ are the standard coordinate vectors in $\R^d$.
Thus, in the case of uncorrelated measurement errors considered here, 
\eqref{equ:Hrep} expresses $\H$  
as a sum of rank-one operators each corresponding to a
specific sensor.  

\begin{remark}\label{rmk:correlated_Hessian}
The representation~\eqref{equ:Hrep} is a consequence of 
assuming a diagonal noise covariance matrix.
In the case $\ncov$ is non-diagonal, 
the off-diagonal entries of the noise covariance matrix introduce interactions between
different sensors and the representation~\eqref{equ:Hrep} is lost.
In that case,
$\ncov^{-1} = \sum_{i,j=1}^d 
(\ncov^{-1})_{ij} e_i \otimes e_j$ and 
\[
    \H=\F^*\ncov^{-1}\F
     =\sum_{i,j=1}^d (\ncov^{-1})_{ij}\, (\F^* e_i) \otimes (\F^* e_j).
\]
That is, 
$\H$ is no longer a sum of 
rank-one contributions from individual sensors.
\end{remark}

\subsection{An update formula and monotonicity of $\Cpost(S)$}\label{sec:update}
Let us consider the impact of adding one sensor on the posterior covariance operator. 
Namely, for any $i \notin S$, we consider the updated posterior covariance operator, 
\[
\Cinc{i}(S) \defeq \Cprior^{1/2} (\I + \Ht(S) + \ft_i \otimes \ft_i)^{-1} \Cprior^{1/2}.
\]
Deriving a convenient expression for $\Cinc{i}(S)$ requires the following 
infinite-dimensional 
analogue of Sherman--Morrison--Woodbury formula for rank-one updates; 
this formula 
will also be needed in Section~\ref{sec:main}.
\begin{lemma}\label{lem:smw_basic}
Let $\A\in\L(\M)$ be a bijection and 
let $u,v\in \M$ and assume $1+\ipm{v}{\A^{-1}u}\neq 0$. Then
$\A+u\otimes v$ is invertible and
\[
(\A+u\otimes v)^{-1}
= \A^{-1} - \frac{\A^{-1}(u\otimes v)\A^{-1}}{1+\ipm{v}{\A^{-1}u}}.
\]
\end{lemma}
\begin{proof}
Since $A \in \L(\M)$ is a bijection, we know by the Banach Open Mapping Theorem 
that it has a bounded inverse. That is, $\A$ is invertible.  
Let 
\[
    \B \defeq \A^{-1} - \frac{\A^{-1}(u\otimes v)\A^{-1}}{1+\ipm{v}{\A^{-1}u}}.
\] 
Note that $\B \in \L(\M)$. 
A direct calculation yields 
$(\A+u\otimes v)\B = \B(\A+u\otimes v) = \I$.
\end{proof}

\begin{remark}
Lemma~\ref{lem:smw_basic} extends the
Sherman--Morrison--Woodbury formula for rank-one updates to a Hilbert space
setting. Such extensions have been considered in much more general settings;
see, e.g.,~\cite{Deng11}.
\end{remark}

The following result provides a convenient formula for $\Cinc{i}$.
\begin{proposition}
If $S \subseteq V$ and $i \notin S$, then
\begin{equation}\label{equ:Cpost_update}
\Cinc{i}(S) = 
\Cpost(S) - 
\frac{\Cprior^{1/2}\tilde z_i \otimes \Cprior^{1/2} \tilde{z}_i }{1 + \ipm{\ft_i}{\tilde z_i}}, 
\qquad
\tilde{z}_i \defeq (\I + \Ht(S))^{-1}\ft_i.
\end{equation}
\end{proposition}

\begin{proof}
Note that by Lemma~\ref{lem:smw_basic}
\[
(\I + \Ht(S) + \ft_i \otimes \ft_i)^{-1} = (\I + \Ht(S))^{-1} - 
\frac{(\I + \Ht(S))^{-1} \ft_i \otimes (\I + \Ht(S))^{-1} \ft_i}{1 + \ipm{\ft_i}{(\I + \Ht(S))^{-1} \ft_i}}.
\]
Applying $\Cprior^{1/2}$ to the expressions on each side 
yields
the desired result.
\end{proof}

The formula~\eqref{equ:Cpost_update} implies the following 
monotonicity property: 
\begin{equation}\label{equ:loewner}
\Cinc{i}(S) \preceq \Cpost(S),\quad \text{for each } S \subset 2^V. 
\end{equation}
Here, $\preceq$ denotes the Loewner order. Specifically, for $\A$ and $\B$ in
$\Lsymp(\M)$, $\A \preceq \B$ means $\B - \A \in \Lsymp(\M)$.
To see~\eqref{equ:loewner}, we simply note that for each $v \in \M$,
\[
 \ipm{v}{(\Cpost(S) - \Cinc{i}(S))v} = 
 \frac{\ipm{v}{\Cprior^{1/2}\tilde z_i}^2}{1 + \ipm{\ft_i}{\tilde z_i}} \geq 0.
\]
It is also immediate to note that, more generally, 
\[
\Cpost(T) \preceq \Cpost(S), \quad \text{whenever}\quad S \subset T. 
\]
Intuitively, this says that additional 
measurements can only reduce uncertainty. 

A more quantitative result can be obtained 
by considering the impact of adding 
a sensor on the trace of the posterior covariance operator.
Namely, 
the reduction in the average posterior variance 
due to adding the sensor indexed by 
$i \notin S$ is  
\[
\trace(\Cpost(S)) - \trace(\Cinc{i}(S))
= 
\frac{\|\Cprior^{1/2}\tilde z_i\|^2}{1 + \ipm{\ft_i}{\tilde z_i}}. 
\]
Analyzing the monotonicity of $\EIG$, which is considered in the next section, 
builds on similar arguments along with properties of the Fredholm determinant.

\subsection{The OED problem}\label{sec:eig_optim}
As noted previously, 
in the present linear Gaussian setting, the expected information 
admits the analytic expression~\eqref{equ:EIG_def}.
Ignoring the factor of $1/2$ in that definition,
we consider
\[
\pEIG := \log\det(\I + \Ht). 
\] 

We consider the OED problem of selecting an optimal 
subset of the candidate sensor locations, where the sensors will be placed,
such that the EIG is maximized. The 
representation of $\Ht$ in~\eqref{equ:Ht_rep} indicates that 
the OED problem can be formulated as that of selecting an 
optimal subset of $\{ \ft_i\}_{i=1}^\Ns$. 
To make matters concrete, let $V = \{1, \ldots, \Ns\}$ index the set of 
candidate sensor locations. We have, 
\begin{equation}\label{equ:EIGdef}
\pEIG(S) \defeq \log\det(\I + \Ht(S)), \quad S \in 2^V. 
\end{equation} 
Subsequently, the OED problem is formulated as 
\begin{equation}\label{equ:optim_binary_eig}
\begin{aligned}
&\max_{S \subset V} \pEIG(S), \\
&\text{subject to } |S| = k.
\end{aligned}
\end{equation}

\section{The expected information gain and its submodularity}\label{sec:main}
In this section, we prove that $\pEIG$ is a monotone submodular function.
We first consider the following technical lemma, which is needed in what follows.
\begin{lemma}\label{lem:logdet}
Let 
$\K$ be a positive self-adjoint trace-class operator on $\M$. Then, for every $u \in \M$,
\[
\log\det(\I + \K + u \otimes u)
=\log\det(\I + \K)
+\log\left(1 + \|(\I + \K)^{-1/2}u\|^2\right).
\]
\end{lemma}
\begin{proof}
Let $\B = \I + \K$.
Note that $\B$ is an invertible operator and
$\B + u \otimes u = \B\bigl(\I + \B^{-1}u \otimes u\bigr)$.
By the multiplicative property~\eqref{equ:mult},
$\det(\B + u \otimes u)
=\det(\B (\I + \B^{-1}u \otimes u ))
=
\det(\B)\,
\det\bigl(\I + \B^{-1}u \otimes u\bigr)$.
Next, note that 
the rank-one operator $\B^{-1}u \otimes u$ has a single nonzero eigenvalue
$\ipm{\B^{-1}u}{u}$. Hence,
\[
\det\bigl(\I + \B^{-1}u \otimes u\bigr)
=
1 + \ipm{\B^{-1}u}{u} 
= 
1 + \ipm{\B^{-1/2}u}{\B^{-1/2} u}.
\]
Therefore,
\[
\det(\B + u \otimes u)
=
\det(\B)\,
\bigl(1 + \|\B^{-1/2}u\|^2\bigr).
\]
The result follows by taking the log of both sides.
\end{proof}

We are now ready to establish monotonicity and submodularity of the 
expected information gain for 
infinite-dimensional Gaussian linear Bayesian inverse problems.
\begin{theorem}\label{thm:submodularity}
The function $\pEIG$ in~\eqref{equ:EIGdef} is monotone and submodular.
\end{theorem}
\begin{proof}
To establish monotonicity, 
it is sufficient to prove that for each $S \in 2^V$,
$\deriv{\pEIG}{S}{r} \geq 0$, for every $r \in V \setminus S$.
Let $S \subseteq V$. By Lemma~\ref{lem:logdet},
for each $r \in V \setminus S$, 
\begin{align}
    \deriv{\pEIG}{S}{r} &= 
    \pEIG(S \cup \{r\}) - \pEIG(S) \notag \\
    &= \log\det(\I + \Ht(S) + \ft_r \otimes \ft_r) - \log\det(\I + \Ht(S)) \notag\\
    &=
    \log\big(1 + \|(\I+\Ht(S))^{-1/2}\ft_r\|^2\big) \geq 0.\label{equ:gain} 
\end{align} 

We next prove $\pEIG$ is submodular. As noted previously, 
this can be done incrementally; cf.~\eqref{equ:incr-submodularity}. Thus, we 
show that 
if $S \subseteq V$ and $r, t$ are distinct elements of $V \setminus S$, then 
\[
\pEIG(S \cup \{ r \}) - \pEIG(S) \geq \pEIG(S \cup \{ r, t\}) - \pEIG(S \cup \{ t\}).
\]
Let $\A = \I + \Ht(S)$ and define 
\[
\upalpha_{ij} \defeq \ipm{\A^{-1}\ft_i}{\ft_j}, \quad i, j \in \{1, \ldots, \Ns\}.
\]
Using this notation and~\eqref{equ:gain}, we have  
\begin{equation}\label{equ:eig_gain_rr}
 \deriv{\pEIG}{S}{r} = \pEIG(S \cup \{r\}) - \pEIG(S)
= \log( 1 + \arr).
\end{equation}
Next, 
note that 
\begin{equation*} 
    \det(\A + \ft_r \otimes \ft_r + \ft_t \otimes \ft_t ) = 
    \big(1 + \ipm{(\A + \ft_t \otimes \ft_t)^{-1}\ft_r}{\ft_r} \big) \det(\A + \ft_t \otimes \ft_t). 
\end{equation*}
Now, by Lemma~\ref{lem:smw_basic},
\begin{equation*} 
(\A+\ft_t \otimes\ft_t)^{-1}
    =\A^{-1} - \frac{(\A^{-1}\ft_t) \otimes (\A^{-1}\ft_t)} 
    {1+\ipm{\A^{-1}\ft_t}{\ft_t}}.
\end{equation*}
Thus, 
\begin{equation}\label{equ:diminishing_returns}
    \ipm{(\A + \ft_t \otimes \ft_t)^{-1}\ft_r}{\ft_r} = \arr - \frac{\art^2}{1 + \att}.
\end{equation}
Therefore, we have 
\[
    0 \leq \det(\A + \ft_r \otimes \ft_r + \ft_t \otimes \ft_t)
    = 
    (1 + \arr - \art^2/(1+\att))\det(\A + \ft_t \otimes \ft_t).
\]
Hence, 
\[
\begin{aligned}
    \pEIG&(S \cup \{ r, t\}) - \pEIG(S \cup \{ t\}) \\
        &= \log\det(\A + \ft_r \otimes \ft_r + \ft_t \otimes \ft_t ) - \log\det(\A + \ft_t \otimes \ft_t)\\
        &= \log(1 + \arr - \art^2/(1+\att))
        \\
        &\leq \log(1 + \arr) \\
        &= \pEIG(S \cup \{r\}) - \pEIG(S). \qedhere
\end{aligned}
\]
\end{proof}

\begin{remark}
Note also that $\pEIG(\emptyset) = 0$, i.e., $\pEIG$ is a normalized function.
\end{remark}

\begin{remark}
The quantity
$\alpha_{rt}^2/(1+\alpha_{tt})$
in~\eqref{equ:diminishing_returns} plays an important role in proving submodularity 
of $\pEIG$.  In the context of sensor placement, 
this quantity explains the reduction in  
the marginal gain from sensor $r$
due to first activating sensor $t$.
\end{remark}

\begin{remark}\label{rmk:submodularity_loss}
The submodularity property of $\pEIG$ does not extend to general linear Gaussian
problems with correlated measurement errors.  This is well-known in the
finite-dimensional setting~\cite{HuanJagalurMarzouk24}.  Intuitively, correlated
measurement errors introduce interactions between sensors; cf.\ 
Remark~\ref{rmk:correlated_Hessian}.  Thus, adding a sensor can increase the
information content of already active sensors, hence violating submodularity. 
\end{remark}

\begin{remark}
In view of Remark~\ref{rmk:submodularity_loss}, 
one might be tempted to apply a whitening
transformation~\cite[Page 80]{KaipioSomersalo05} to the observation
model~\eqref{equ:model} to obtain an identity as the error covariance
matrix.  Namely, applying $\ncov^{-1/2}$ to both sides
of~\eqref{equ:model}, we obtain 
\[ \bar{y} = \bar\F m + \bar\eta, \] with  $\bar
y = \ncov^{-1/2} y$, $\bar\F = \ncov^{-1/2}\F$, and $\bar\eta = \ncov^{-1/2}\eta
\sim \GM{0}{\I}$. Note, however, that in this case the entries of $\bar y$ are
no longer associated with specific sensors.  Thus, while applying a whitening
transformation results in diagonal noise covariance matrix, it fundamentally
changes the structure of the combinatorial problem: in this case sensor
selection is not equivalent to picking a subset of the entries of $\bar y$.
\end{remark}

\section{Computational considerations}\label{sec:lazy_dopt}
Consider the greedy method in Algorithm~\ref{alg:greedy_submodular} applied 
to the problem~\eqref{equ:optim_binary_eig}. The major 
cost is incurred when computing the marginal gains in line 3 of the algorithm.
Specifically, at the $l$th step of the algorithm, we need to compute 
\begin{equation}\label{equ:EIG_gain}
\deriv{\pEIG}{S_{l-1}}{v}, \quad \text{for every } v \in V \setminus S_{l-1}. 
\end{equation}
As noted in the proof of Theorem~\ref{thm:submodularity} (cf.~\eqref{equ:eig_gain_rr}),
\begin{equation}\label{equ:marginal_gain}
\deriv{\pEIG}{S_{l-1}}{v} = \log\big(1 + \ipm{(\I + \Ht(S_{l-1}))^{-1}\ft_v}{\ft_v}\big).
\end{equation} 

In problems where $d$ is not too large, and we can afford to precompute and 
store $\{\ft_i\}_{i=1}^\Ns$, the greedy procedure can be carried out without any
PDE solves.  Such a precomputation step requires computing $\Cprior^{1/2}\F^*
e_i$, $i \in \{1, \ldots, \Ns\}$. In inverse problems governed by PDEs, this
amounts to $\Ns$ adjoint PDE solves. In some ill-posed inverse problem 
$\F$ might have a rapidly decaying spectrum, and the numerical rank of 
$\F$ might be considerably smaller than $d$, especially if $d$ is large. 
The smoothing properties of $\Cprior$ typically result in 
faster spectral decay for $\F \Cprior^{1/2}$; 
see~\cite{AlexanderianPetraStadlerEtAl14,AlexanderianSaibaba18}. In such cases, 
this operator can be approximated accurately and efficiently 
by a low-rank approximation. Hence, 
one can accelerate computations further by precomputing a low-rank approximation 
of $\F \Cprior^{1/2}$, which can then be used to compute highly accurate 
approximations of $\{\ft_i\}_{i=1}^\Ns$. This approach is very much related 
to developments in~\cite{AlexanderianPetraStadlerEtAl14,AlexanderianSaibaba18}
that considered A- and D-optimal design of experiments for infinite-dimensional 
Bayesian linear inverse problems, albeit within a convex relaxation 
and gradient-based optimization framework. 

We note, however, that even with the simplification considered above, the
cumulative cost of applying $(\I + \Ht(S_{l-1}))^{-1}$ to elements of $V
\setminus S_{l-1}$ can be high, if $\Ns$ is large and many greedy steps are
needed.  In Section~\ref{sec:inside_out}, we present an alternative formula for
the marginal gain that involves the inverse of an operator on the measurement
space.  This is beneficial in many practical inverse problems in which, upon
discretization, the discretized parameter dimension is very large.

In Section~\ref{sec:lazy}, we consider an efficient 
implementation of the greedy method for maximizing the EIG---the \emph{lazy} greedy 
procedure. While lazy greedy is well-known in computer science and discrete 
optimization~\cite{KrauseGolovin14,Minoux78}, it remains underutilized in the
PDE-constrained optimization and inverse problems communities. Combined with the
idea of low-rank approximation discussed earlier and/or the
measurement-space approach considered in Section~\ref{sec:inside_out}, this can
lead to highly efficient implementations of the greedy method for finding
near optimal sensor placements.

\subsection{A measurement space approach to computing the marginal gains}\label{sec:inside_out}
Let $S = \{s_1, \ldots, s_k\} \subset V$. Computing the marginal 
gains needed within the greedy procedure requires computing expressions of the 
form 
$(\I + \Ht(S))^{-1}m$, for $m \in \M$. Here, we present a ``measurement space approach''
for computing such expressions that involves only inverses of operators 
on $\R^k$. 
It is straightforward to 
note that $\Ht(S) = \Ft(S)^* \Ft(S)$, where $\Ft(S):\M \to \R^k$ is defined by 
\[
    \Ft(S) m \defeq 
    \begin{bmatrix} 
        \ipm{\ft_{s_1}}{m} \\ \vdots \\\ipm{\ft_{s_k}}{m}
    \end{bmatrix}.
\]
Thus, $(\I + \Ht(S))^{-1}m = (\I + \Ft(S)^*\Ft(S))^{-1}m$.
The following 
technical result, which 
provides a Sherman--Morrison--Woodbury-like formula in a Hilbert 
space setting, provides an efficient way of computing such 
expressions.  
\begin{lemma}\label{lem:basic}
Let $\G:\M \to \mathcal{K}$ be a bounded linear transformation between 
Hilbert spaces $\M$ and $\mathcal{K}$. Then 
$(\I+\G^*\G)^{-1} = \I - \G^*(\I+\G\G^*)^{-1}\G$. 
\end{lemma}
\begin{proof}
See Appendix~\ref{appdx:lemma_proof}.
\end{proof}
Using this result, we have that for each $m \in \M$,
\[
\begin{aligned}
(\I + \Ht(S))^{-1}m 
&= (\I + \Ft(S)^* \Ft(S))^{-1}m \\
&= m - \Ft(S)^*(\I + \Ft(S)\Ft(S)^*)\Ft(S) m.
\end{aligned}
\]
Thus, for each $m \in \M$, 
\[
\begin{aligned}
\ipm{(\I + \Ft(S)^* \Ft(S))^{-1}m}{m} &= \ipm{m - \Ft(S)^*(\I + \Ft(S)\Ft(S)^*)^{-1}\Ft(S) m}{m} \\
&= \| m\|^2 - \ipm{(\I + \Ft(S)\Ft(S)^*)^{-1}\Ft(S) m}{\Ft(S) m}.
\end{aligned}
\]
These calculations lead to an alternate formula for computing 
the marginal gains in~\eqref{equ:marginal_gain}. Namely, for $S \subset V$, 
where $S$ has cardinality $1 \leq k < \Ns$, 
\begin{align}
\deriv{\pEIG}{S}{v} &= 
\log\big(1 + \ipm{(\I + \Ht(S))^{-1}\ft_v}{\ft_v}\big) \label{equ:marginal_gain_orig}\\
&=
\log\big(1 + \| \ft_v\|^2 - \ipm{(\I + \Ft(S)\Ft(S)^*)^{-1}\Ft(S) \ft_v}{\Ft(S) \ft_v} \big), 
\label{equ:marginal_gain_alt}
\end{align}
for all $v \in V \setminus S$. The advantage of \eqref{equ:marginal_gain_alt} 
over \eqref{equ:marginal_gain_orig}
is that it involves the inverse of an operator on 
$\R^k$ as opposed to the inverse of an operator on $\M$. 

The present measurement-space formulation for computing the
marginal gain in $\pEIG$ is related to a similar approach in  
\cite{AlexanderianPetraStadlerEtAl21}.  In that article, a measurement-space
approach was used to compute a variant of the Bayesian A-optimality criterion,
in the context of OED under model uncertainty and in a discretized
framework.

\subsection{Exploiting submodularity via a lazy greedy procedure}\label{sec:lazy}
The lazy greedy method~\cite{Minoux78} uses submodularity to reduce the number
of marginal gain evaluations.  Note that
by submodularity of $\pEIG$, 
\[
\deriv{\pEIG}{S_{l-1}}{v} \geq \deriv{\pEIG}{S_{l}}{v}, \quad l \geq 1.
\]
This suggests that instead of naively computing all the
requisite marginal gains in~\eqref{equ:EIG_gain} at each step of the greedy
algorithm, we can use the already computed marginal gains as an upper bound for
the subsequent ones. This is the central idea behind the lazy greedy method. 

Letting $S_0 = \emptyset$, the first iteration of the lazy greedy method is 
the same as standard greedy: we compute 
\[
\dhat(v) \defeq \deriv{\pEIG}{S_{0}}{v} = \pEIG(\{v\}),
\]
for every $v \in V$. Then, we choose $s_1 \in V$ that maximizes $\dhat(v)$ and
let $S_1 = \{s_1\}$. The first step of the lazy greedy method is concluded by 
sorting the values of $\{\dhat(v)\}_{v \in V \setminus S_1}$ in descending
order.  The sorted set of marginal gains will be maintained and updated in the
subsequent iterations of the lazy greedy method.  
\begin{figure}[htbp]
\centering
\fbox{
    \begin{minipage}{1\textwidth}
\begin{enumerate}[label=(\roman*)]
\smallskip
\item Choose $v^* \in V \setminus S_{l-1}$ according to
\[
v^* \in \argmax_{v \in V \setminus S_{l-1}} \dhat(v);
\]
note that $\{ \dhat(v)\}_{v \in V \setminus S_{l-1}}$
are available from step $l-1$.
\smallskip
\item Compute $\deriv{\pEIG}{S_{l-1}}{v^*}$ and update
\[
\dhat(v^*) = \deriv{\pEIG}{S_{l-1}}{v^*}.
\]
\item If $v^*$ still maximizes $\{\dhat(v)\}_{v \in V \setminus S_{l-1}}$,
then let $S_l = S_{l-1}\cup\{v^*\}$ and go to step $l+1$ of the
lazy greedy procedure; otherwise, go to (i).
\smallskip
\end{enumerate}
    \end{minipage}
}
\caption{One step of the lazy greedy algorithm.} 
\label{fig:lazy}
\end{figure}
The $l$th step of the lazy greedy method is outlined in Figure~\ref{fig:lazy}. 
Note that since $\pEIG$ is submodular, the previously computed marginal gains
provide upper bounds on future marginal gains.  This justifies the
acceptance criterion in step (iii).

The lazy greedy method preserves the approximation guarantee of the standard
greedy algorithm.  In practice, lazy evaluations often yield
substantial reductions in computational cost by avoiding unnecessary marginal
gain evaluations; see, e.g.,~\cite{BadanidiyuruVondrak2014,KrauseSinghGuestrin08,LeskovecKrauseGuestrinEtAl07}.  
Modern submodular optimization libraries often implement lazy greedy as a standard
option~\cite{SchreiberBilmesNoble2020,KaushalRamakrishnanIyer2022}.  
One can further reduce the number of 
marginal gain evaluations by combining lazy evaluations with 
the stochastic greedy method~\cite{MirzasoleimanBadanidiyuruKarbasiEtAl15}.

\subsection{Improving the greedy solution}\label{sec:exchange-algorithms}
In practice, one can utilize an exchange algorithm~\cite{Federov1972:1} after
the greedy algorithm to further improve the solution.  An exchange algorithm
takes an existing solution and swaps out elements if doing so improves the
quality of the solution.  Usually the swap that yields the largest improvement
in the objective value is made.  The exchange algorithm terminates if no swap
improves the objective value.  The procedure is outlined in
Algorithm~\ref{alg:exchange_algorithm}.  Note that because a swap is only made
if it improves the solution, we preserve the approximation
guarantee~\eqref{eq:approximation-ratio}.  
See~\cite{LauZhou:2022:1} for the analysis of the exchange algorithm for a
related D-optimal design problem.

\begin{algorithm}[H]
\caption{Exchange Algorithm}
\label{alg:exchange_algorithm}
\begin{algorithmic}[1]
\REQUIRE Ground set $V$, set function $f : 2^V \to \R$, and initial set $S_0\subset V$ with $|S_0|=k$
\ENSURE Improved solution $S$
\STATE $l = 0$
\WHILE{$\exists u \in S_l,\ v \in V\setminus S_l$ such that
$f((S_l\cup\{v\})\setminus\{u\})>f(S_l)$}
\STATE $\displaystyle
(u^\ast,v^\ast)
= \argmax_{u \in S_l,\ v \in V \setminus S_l}
f((S_l\cup\{v\})\setminus\{u\})$
\STATE $S_{l+1} = (S_l \cup \{v^\ast\})\setminus\{u^\ast\}$
\STATE $l = l + 1$
\ENDWHILE
\RETURN $S = S_l$
\end{algorithmic}
\end{algorithm}

\section{Conclusion}\label{sec:conc}
We have proven the submodularity of the expected information
gain (EIG) for infinite-dimensional Bayesian linear inverse problems. This
provides a rigorous foundation for structure-exploiting computational
methods for greedy sensor placement, as detailed in our discussion 
of computational considerations. 

Extending the result on submodularity of the EIG to the infinite-dimensional setting
implies that the approximation guarantee of the greedy method for maximizing the
EIG for the class of inverse problems under study is discretization invariant.  
Specifically, the approximation guarantee holds in the infinite-dimensional
setting and is not merely a byproduct of discretization.  This is also relevant
to variants of the greedy algorithm for OED problems with more complex
constraints, such as matroid or knapsack constraints.  

The proof of submodularity relies on a sensor-wise decomposition of the
prior-preconditioned data misfit Hessian into rank-one terms, each corresponding
to an individual sensor. This decomposition was possible due to the assumption
of uncorrelated measurement errors.  Another important insight was reducing the
proof of submodularity to the analysis of one-element increments.

Our discussion of computational considerations indicates that leveraging the
structural properties of the inverse problem and the submodularity of the EIG
leads to highly efficient implementations of the greedy procedure. This has
broad implications for optimal sensor placement in infinite-dimensional Bayesian
linear inverse problems governed by PDEs.  Algorithms for fast sensor placement
are particularly critical in emerging computational paradigms such as digital
twins~\cite{Asch22}, where real-time decision-making and model updates are essential.

The present study is also relevant to nonlinear inverse problems, where a 
common approach is to use linearized formulations. Namely, by 
linearizing the parameter-to-observable map about a preliminary parameter estimate, 
the problem may be reduced to a setting for which the approximation 
guarantee of greedy sensor placement holds. This provides a principled 
approach for adaptive OED strategies, where the sensor 
network is iteratively refined as the parameter estimates are updated.

\section*{Acknowledgments} 
We would like to thank the two anonymous referees, whose thoughtful suggestions 
resulted in a much improved article. 
\bibliographystyle{AIMS} 
\bibliography{refs}

@article{DashtiLawStuart13,
  title={MAP estimators and their consistency in {B}ayesian nonparametric inverse problems},
  author={Dashti, Masoumeh and Law, Kody JH and Stuart, Andrew M and Voss, Jochen},
  journal={Inverse Problems},
  volume={29},
  number={9},
  pages={095017},
  year={2013},
  publisher={IOP Publishing}
}

@book{Ucinski05,
    AUTHOR = {Uci\'{n}ski, Dariusz},
     TITLE = {Optimal Measurement Methods for Distributed Parameter System
              Identification},
    SERIES = {Systems and Control Series},
 PUBLISHER = {CRC Press, Boca Raton, FL},
      YEAR = {2005},
     PAGES = {xviii+371},
      ISBN = {0-8493-2313-4},
   MRCLASS = {93-02 (62M09 62M30 62M45 90B80 93C20 93E12 94A12)},
  MRNUMBER = {2465165},
MRREVIEWER = {Ewaryst Rafaj\l owicz},
}

@article{Kullback:1951,
  title={On information and sufficiency},
  author={Kullback, S. and Leibler, R.A.},
  journal={The Annals of Mathematical Statistics},
  volume={22},
  number={1},
  pages={79--86},
  year={1951},
}

@article{AlexanderianGloorGhattas16,
  Title                    = {On {B}ayesian {A}-and {D}-optimal experimental designs in infinite dimensions},
  Author                   = {Alen Alexanderian and Philip J. Gloor and Omar Ghattas},
  Journal                  = {Bayesian Analysis},
  Year                     = {2016},
  Volume                   = {11},
  Number                   = {3},
  Pages                    = {671--695},
  Doi                      = {10.1214/15-BA969},
}

@article{Alexanderian21,
  title={Optimal experimental design for infinite-dimensional {B}ayesian inverse problems governed by {PDEs}: A review},
  author={Alexanderian, Alen},
  journal={Inverse Problems},
  volume={37},
  number={4},
  pages={043001},
  year={2021},
  publisher={IOP Publishing}
}

@article{AlexanderianSaibaba18,
  title={Efficient {D}-optimal design of experiments for infinite-dimensional {B}ayesian linear inverse problems},
  author={Alexanderian, Alen and Saibaba, Arvind K.},
  journal={SIAM Journal on Scientific Computing},
  volume={40},
  number={5},
  pages={A2956--A2985},
  year={2018},
  publisher={SIAM}
}

@article{AlexanderianPetraStadlerEtAl14,
  Title                    = {{A}-optimal design of experiments for infinite-dimensional {B}ayesian linear inverse problems with regularized $\ell_0$-sparsification},
  Author                   = {Alen Alexanderian and Noemi Petra and Georg Stadler and Omar Ghattas},
  Journal                  = {SIAM Journal on Scientific Computing},
  Year                     = {2014},
  Number                   = {5},
  Pages                    = {A2122--A2148},
  Volume                   = {36},
  Doi                      = {10.1137/130933381}
}

@article{KrauseGolovin14,
  title={Submodular function maximization.},
  author={Krause, Andreas and Golovin, Daniel},
  journal={Tractability},
  volume={3},
  number={71-104},
  pages={3},
  year={2014}
}

@inproceedings{LeskovecKrauseGuestrinEtAl07,
  author    = {Leskovec, Jure and Krause, Andreas and Guestrin, Carlos and
               Faloutsos, Christos and VanBriesen, Jeanne and Glance, Natalie},
  title     = {Cost-Effective Outbreak Detection in Networks},
  booktitle = {Proceedings of the 13th ACM SIGKDD International Conference on
               Knowledge Discovery and Data Mining},
  pages     = {420--429},
  year      = {2007},
  publisher = {ACM},
  doi       = {10.1145/1281192.1281239},
}

@InProceedings{Minoux78,
author="Minoux, Michel",
editor="Stoer, J.",
title="Accelerated Greedy Algorithms for Maximizing Submodular Set Functions",
booktitle="Optimization Techniques",
year="1978",
publisher="Springer Berlin Heidelberg",
pages="234--243"}

@article{NemhauserWolseyFisher78,
  title={An analysis of approximations for maximizing submodular set functions—{I}},
  author={Nemhauser, George L and Wolsey, Laurence A and Fisher, Marshall L},
  journal={Mathematical Programming},
  volume={14},
  pages={265--294},
  year={1978},
  publisher={Springer}
}

@article{LauZhou:2022:1,
  title={A local search framework for experimental design},
  author={Lau, Lap Chi and Zhou, Hong},
  journal={SIAM Journal on Computing},
  volume={51},
  number={4},
  pages={900--951},
  year={2022},
  publisher={SIAM}
}

@inproceedings{MirzasoleimanBadanidiyuruKarbasiEtAl15,
  title={Lazier than lazy greedy},
  author={Mirzasoleiman, Baharan and Badanidiyuru, Ashwinkumar and Karbasi, Amin and Vondr{\'a}k, Jan and Krause, Andreas},
  booktitle={Proceedings of the AAAI Conference on Artificial Intelligence},
  year={2015}
}

@InProceedings{ShamaiahBanerjeeVikalo10,
author={Shamaiah, Manohar and Banerjee, Siddhartha and Vikalo, Haris},
booktitle={49th IEEE Conference on Decision and Control (CDC)}, 
title={Greedy sensor selection: Leveraging submodularity}, 
year={2010},
pages={2572--2577},
doi={10.1109/CDC.2010.5717225}}

@book{Schrijver03,
    address = {Berlin Heidelberg New York},
    series = {Algorithms and combinatorics},
    title = {Combinatorial Optimization: Polyhedra and Efficiency},
    isbn = {978-3-540-44389-6},
    shorttitle = {Combinatorial optimization},
    language = {eng},
    number = {24},
    publisher = {Springer},
    author = {Schrijver, Alexander},
    year = {2003},
}

@article{KrauseSinghGuestrin08,
  author  = {Krause, Andreas and Singh, Ajit and Guestrin, Carlos},
  title   = {Near-Optimal Sensor Placements in {G}aussian Processes: Theory, Efficient Algorithms and Empirical Studies},
  journal = {Journal of Machine Learning Research},
  volume  = {9},
  pages   = {235--284},
  year    = {2008},
  doi     = {10.5555/1390681.1390689},
}

@article{KelmansKimelfeld83,
  title={Multiplicative submodularity of a matrix's principal minor as a function of the set of its rows and some combinatorial applications},
  author={Kelmans, Alexander K and Kimelfeld, Boris N},
  journal={Discrete Mathematics},
  volume={44},
  number={1},
  pages={113--116},
  year={1983},
  publisher={Elsevier}
}

@book{ManningSchutze99,
  title={Foundations of Statistical Natural Language Processing},
  author={Manning, Christopher and Schutze, Hinrich},
  year={1999},
  publisher={MIT press}
}

@Article{Stuart10,
  Title                    = {Inverse problems: {A B}ayesian perspective},
  Author                   = {Andrew M. Stuart},
  journal                  = {Acta Numerica},
  Year                     = {2010},
  Pages                    = {451-559},
  Volume                   = {19}
}

@article{Simon77,
  title={Notes on infinite determinants of Hilbert space operators},
  author={Simon, Barry},
  journal={Advances in Mathematics},
  volume={24},
  number={3},
  pages={244--273},
  year={1977},
  publisher={Elsevier}
}

@article{BoydSnigireva19,
  title={On the analyticity of the {F}redholm determinant},
  author={Boyd, Christopher and Snigireva, Nina},
  journal={Monatshefte f{\"u}r Mathematik},
  volume={190},
  number={4},
  pages={675--687},
  year={2019},
  publisher={Springer}
}

@book{Simon05,
  author    = {Simon, Barry},
  title     = {Trace Ideals and Their Applications},
  edition   = {2},
  publisher = {American Mathematical Society},
  year      = {2005},
  series    = {Mathematical Surveys and Monographs},
  volume    = {120},
  address   = {Providence, RI}
}

@article{EswarRaoSaibaba2026,
  author  = {Eswar, Srinivas and Rao, Vishwas and Saibaba, Arvind K.},
  title   = {Bayesian {D}-Optimal Experimental Designs via Column Subset Selection},
  journal = {SIAM Journal on Scientific Computing},
  volume  = {48},
  number  = {2},
  pages   = {A905--A928},
  year    = {2026},
  doi     = {10.1137/24M1649861},
}

@article{AlexanderianDiazRaoEtAl2026,
  author  = {Alexanderian, Alen and D{\'i}az, Hugo and Rao, Vishwas and Saibaba, Arvind K.},
  title   = {Optimal Sensor Placement under Model Uncertainty in the Weak-Constraint {4D-Var} Framework},
  journal = {Foundations of Data Science},
  volume  = {11},
  pages   = {1--28},
  year    = {2026},
  doi     = {10.3934/fods.2026002},
}

@article{Deng11,
  title={A generalization of the {S}herman--{M}orrison--{W}oodbury formula},
  author={Deng, Chun Yuan},
  journal={Applied Mathematics Letters},
  volume={24},
  number={9},
  pages={1561--1564},
  year={2011},
  publisher={Elsevier}
}

@article{AlexanderianPetraStadlerEtAl21,
  title={Optimal design of large-scale {Bayesian} linear inverse problems under reducible model uncertainty: Good to know what you don't know},
  author={Alexanderian, Alen and Petra, Noemi and Stadler, Georg and Sunseri, Isaac},
  journal={SIAM/ASA Journal on Uncertainty Quantification},
  volume={9},
  number={1},
  pages={163--184},
  year={2021},
  publisher={SIAM}
}

@book{ReedSimon80,
  author    = {Reed, Michael and Simon, Barry},
  title     = {Methods of Modern Mathematical Physics. I: Functional Analysis},
  publisher = {Academic Press},
  year      = {1980}
}

@book{MacCluer09,
  title={Elementary Functional Analysis},
  author={MacCluer, Barbara D},
  volume={253},
  year={2009},
  publisher={Springer}
}

@book{Asch22,
  title={A Toolbox for Digital Twins: From Model-Based to Data-Driven},
  author={Asch, Mark},
  year={2022},
  publisher={SIAM}
}

@book{KaipioSomersalo05,
  author    = {Kaipio, Jari P. and Somersalo, Erkki},
  title     = {Statistical and Computational Inverse Problems},
  publisher = {Springer},
  year      = {2005},
  doi       = {10.1007/b138659},
}

@article{HuanJagalurMarzouk24,
  title={Optimal experimental design: Formulations and computations},
  author={Huan, Xun and Jagalur, Jayanth and Marzouk, Youssef},
  journal={Acta Numerica},
  volume={33},
  pages={715--840},
  year={2024},
  publisher={Cambridge University Press}
}

@book{Federov1972:1,
	title = {Theory of {Optimal} {Experiments}},
	language = {en},
	publisher = {Academic Press},
	author = {Fedorov, V V},
	year = {1972},
}

@inproceedings{BadanidiyuruVondrak2014,
  author    = {Badanidiyuru, Ashwinkumar and Vondr{\'a}k, Jan},
  title     = {Fast algorithms for maximizing submodular functions},
  booktitle = {Proceedings of the Twenty-Fifth Annual ACM-SIAM Symposium on
               Discrete Algorithms},
  pages     = {1497--1514},
  publisher = {SIAM},
  year      = {2014},
  doi       = {10.1137/1.9781611973402.110}
}

@article{SchreiberBilmesNoble2020,
  author  = {Schreiber, Jacob and Bilmes, Jeff and Noble, William Stafford},
  title   = {apricot: Submodular selection for data summarization in Python},
  journal = {Journal of Machine Learning Research},
  volume  = {21},
  number  = {161},
  pages   = {1--6},
  year    = {2020}
}

@misc{KaushalRamakrishnanIyer2022,
  author        = {Kaushal, Vishal and Ramakrishnan, Ganesh and Iyer, Rishabh},
  title         = {{Submodlib}: A Submodular Optimization Library},
  year          = {2022},
  note          = {arXiv preprint arXiv:2202.10680},
  eprint        = {2202.10680},
  archivePrefix = {arXiv},
  primaryClass  = {cs.LG},
}

\appendix

\section{Proof of Lemma~\ref{lem:basic}}\label{appdx:lemma_proof}
\begin{proof}
Note that $\G\G^*$ and $\G^*\G$ belong to $\Lsymp(\M)$ and $\Lsymp(\mathcal{K})$,
respectively. 
It is straightforward to note that $\I + \G^*\G$ and 
$\I + \G \G^*$ are both invertible.
Furthermore, 
\begin{align*}
  (\I+\G^*\G)&\bigl(\I - \G^*(\I+\G\G^*)^{-1}\G\bigr)\\ 
  &= \I - \G^*(\I+\G\G^*)^{-1}\G + \G^*\G - \G^*\G\,\G^*(\I+\G\G^*)^{-1}\G \\
  &= \I + \G^*\G - \G^*(\I+\G\G^*)^{-1}\G - \G^*(\G\G^*)(\I+\G\G^*)^{-1}\G \\
  &= \I + \G^*\G - \G^*(\I+\G\G^*)(\I+\G\G^*)^{-1}\G\\ 
  &= \I + \G^*\G - \G^*\G\\ 
  &= \I. \qedhere
\end{align*}
\end{proof}

\end{document}